%% file: main.tex
\documentclass[letterpaper, 10 pt, conference]{ieeeconf}
\IEEEoverridecommandlockouts
\usepackage{graphicx} 
\usepackage[hidelinks]{hyperref}
\usepackage{dsfont}
\usepackage{enumerate}%
\usepackage{cite}
\usepackage{balance}
\usepackage{pgfplots}
\usepackage{pgfplotstable}
\pgfplotsset{compat=1.18}

\usepackage{amsmath,amsthm,amssymb}

\usepackage{subcaption}

\usepackage{tikz}
\usetikzlibrary{arrows.meta, positioning}

\usepackage{comment}

\usepackage[capitalize]{cleveref}
\Crefname{lemma}{Lemma}{Lemmas}
\Crefname{proposition}{Proposition}{Propositions}
\Crefname{assumption}{Assumption}{Assumptions}

\newtheorem{proposition}{Proposition}
\newtheorem{lemma}{Lemma}
\newtheorem{theorem}{Theorem}

\theoremstyle{definition}
\newtheorem{definition}{Definition}
\newtheorem{assumption}{Assumption}
\newtheorem{remark}{Remark}

\newtheorem{example}{Example}

\newcommand{\R}{\mathds{R}}

\newcommand{\V}{\mathds{V}}
\newcommand{\spanop}{\operatorname{span}}
\newcommand{\diam}{\operatorname{diam}}
\newcommand{\ran}{\operatorname{ran}}

\usepackage{soul}
\usepackage{bm}
\newcommand{\vb}[1]{\boldsymbol{#1}} 
\newcommand{\mb}[1]{\mathbf{#1}}     

\title{\LARGE \bf
Limitations of LTI Koopman Modeling for Nonlinear Control Systems
}

\author{Johannes Heeg and Karl Worthmann%
\thanks{The authors are with the Optimization-based Control Group, Technische Universität Ilmenau, 98693 Ilmenau, 
Germany {\tt\small [johannes.heeg, karl.worthmann]@tu-ilmeanu.de}.
The authors are thankful for the support by the Deutsche Forschungsgemeinschaft (DFG, German Research Foundation) within the research unit ALeSCo~--~\textbf{A}ctive \textbf{Le}arning in \textbf{S}ystems and \textbf{Co}ntrol, project-ID 535860958.}
}

\begin{document}

\maketitle
\thispagestyle{empty}
\pagestyle{empty}

\begin{abstract}
    Koopman operator theory yields powerful tools for modeling, analysis, and control of nonlinear dynamical systems. 
    Prominently, linear time-invariant (LTI) Koopman representations have been proposed to enable the application of linear control techniques, such as LQR and convex MPC. 
    In this work, we investigate the implications of exact LTI Koopman representations for continuous-time nonlinear control systems. 
    In particular, we show that, assuming a mild controllability condition and the inclusion of the coordinate maps, the dynamics of the underlying control system must be affine linear. 
    Furthermore, we study the modeling bias introduced by the LTI structure and analyze its dependency on the choice of observables.
\end{abstract}

\section{Introduction}

Koopman operator theory
has attracted considerable attention due to
its ability to provide
a linear perspective onto nonlinear dynamical
systems~\cite{mezic2005spectral,rowley2009spectral}.
In the Koopman framework, 
the evolution of a nonlinear system is described through the lens of so-called observable functions (observables for short), rendering the respective dynamics linear but, in general, infinite dimensional.
Recently, different numerical methods to approximate the Koopman operator and its spectrum on finite-dimensional subspaces have been proposed, see, e.g., \cite{Colb24:multiverse} and the references therein.

The successful application of Koopman-based modeling  
for autonomous systems inspired extensions to nonlinear control systems~\cite{surana2016koopman,proctor2018generalizing,williams2016a,haseli2026modeling,fan2026}, including an in-depth error analysis~\cite{BoldPhil25} and robust controller design~\cite{STRASSER2026112732,shang2025exponential}; see the overview article~\cite{StraWort26}.
Prominently, LTI Koopman modeling, i.e., modeling nonlinear control systems as lifted linear time-invariant systems, has emerged as a successful paradigm enabling the use of linear controller design, like linear quadratic regulation~\cite{bevanda2022} or convex model predictive control~\cite{korda2018,asada2024}, for nonlinear systems.
While for particular 
systems, a finite-dimensional LTI Koopman representation can be constructed~\cite{proctor2018generalizing}, it is known that such representations do not exist for general nonlinear systems~\cite{bruder2021,iacob2024,StraWort26}. 
This insight motivated the search for necessary and sufficient conditions for their existence~\cite{bruder2021,iacob2024} and resulted in invariance conditions~\cite{brunton2016koopman,KordMezi20:optimal}.

In this paper, we prove that, if an exact LTI Koopman representation exists, the nonlinear control system must be affine linear assuming the common inclusion of coordinate maps and a rather mild controllability condition. This condition, called \textit{directional excitability} in the following, provides a new structural insight into the admissible system class.
We show that a controllable linearization as well as (local) satisfaction of the Lie-algebra rank condition~(LARC) imply directional excitability, 
substantiating our claim of that being a rather mild condition.
Furthermore, we investigate the LTI modeling bias for nonlinear control systems with constant input matrix. 
Our contribution w.r.t.\ that are conditions derived from an analysis of the interplay of the LTI modeling bias with the choice of observables and the system dynamics.
In addition, we provide analytical and numerical examples to illustrate our findings.

Concurrent with our work, \cite{shang2026existence}~studied the existence of exact finite-dimensional LTI Koopman representations for discrete-time control systems. 
In comparison, our work considers continuous-time systems, requiring the analysis of the Koopman generator,
avoids 
assuming finite dimensionality a priori, and adds a novel error analysis. Our results indicate that the inclusion of additional observables cannot reduce the modeling error arbitrarily and gives a new characterization based on directional excitability.

The remainder of this paper is organized as follows.
\Cref{sec:problem_setting} introduces the problem formulation. Our characterization of systems admitting exact LTI Koopman representations is presented in \Cref{sec:main}. In \Cref{sec:bias}, we present our results on the LTI modeling bias, before conclusions are drawn.\\

\noindent\textbf{Notation}:
The inner product on $\R^n$ is denoted by $\langle \cdot, \cdot \rangle$, and $\|\cdot\|$ is the 2-norm or its induced matrix norm. $\diam(X)$ denotes the diameter of a bounded set~$X$. For a subset~$M \subset \R^n$, $M^\perp$ denotes its orthogonal complement, i.e., $M^\perp = \{ x \in \R^n \mid \langle x,y \rangle = 0\ \forall\,y \in M \}$.
$\operatorname{Id}$ is the identity operator.
For a matrix~$\mb A$, the range is denoted by~$\ran \mb A$ and the Moore-Penrose pseudoinverse by $\mb A^\dagger$.
$L^2(\R^n)$ is the space of square-integrable functions, $\mathcal{C}^k(\R^n)$ the space of $k$-times continuously differentiable functions, 
and $L^\infty_{\operatorname{loc}}([0,\infty),\R^m)$ the space of locally bounded functions
from~$[0, \infty)$ to~$\R^m$.

\section{Problem Setting}
\label{sec:problem_setting}

In this section, we recap  
Koopman operator theory for control systems and define LTI Koopman representations of a system w.r.t.\ a 
function space.
Then, assuming the existence of an LTI Koopman representation including the coordinate maps, we derive a first structural property of the nonlinear system dynamics
\begin{equation}\label{eq:sys}
    \dot{\vb{x}}(t) = \tilde{\vb{f}}(\vb{x}(t), \vb{u}(t)).
\end{equation}
$\vb{x}(t) \in \R^n$ and $\vb{u}(t) \in \R^m$ are the state and the 
input at time $t \geq 0$, respectively. 
Assuming continuous differentiability of $\tilde{\vb{f}} \in \mathcal{C}(\R^n \times \R^m,\R^n)$ w.r.t.\ its first argument, $\vb{x}(t;\vb{x}^0,\vb{u})$ denotes the unique (local) solution 
for the initial condition $\vb{x}(0) = \vb{x}^0$ and control input $\vb{u}\in L^\infty_{\operatorname{loc}}([0,\infty),\R^m)$.

Following the concept of the Koopman control family~\cite{haseli2026modeling}, we define the Koopman operator~$\mathcal{K}^t_{\vb{u}}$ for a constant control function $\vb u(t) \equiv \vb{u} \in \R^m$ and a (sufficiently small) time step~$t$ 
by the identity $\mathcal{K}^t_{\vb{u}} \circ \psi = \psi(\vb x(t;\cdot),\vb u)$ for all $\psi \in L^2(\R^n)$. 
The Koopman generator~$\mathcal{L}_{\vb{u}}$ is defined on the domain $\{ \psi \in L^2(\R^n) \mid \lim_{t \searrow 0} \frac 1t (\mathcal{K}^t_{\vb{u}} - \operatorname{Id}) \psi \text{ exists} \}$. The domain encompasses continuously differentiable 
observables, i.e., functions $\psi \in \mathcal{C}^1(\R^n)$, for which the identity for the Koopman generator~$\mathcal{L}_{\vb{u}}$ reads\footnote{We use $\mathcal{C}^1(\R^n)$ to streamline the presentation, while we formally would have to restrict ourselves either to open and bounded sets $X \subset \R^n$ to ensure $\mathcal{C}^1(X) \subset L^2(X)$ or use a weighted $L^2$-space.}
\begin{equation}\label{eq:generator}
    (\mathcal{L}_{\vb u} \psi)(\vb{x}) = \langle \nabla \psi(\vb{x}), \tilde{\vb{f}}(\vb{x}, \vb{u}) \rangle.
\end{equation}
The Koopman generator and the Koopman operator are linear operators acting on infinite-dimensional function spaces.

For modeling and control, one resorts to finite-dimensional approximations on the space $\spanop (\{ \psi_1, \dots, \psi_N \}) \subset L^2(\R^n)$
spanned by the $N$ observables~$\psi_1, \ldots, \psi_N$.
Then, an LTI approximation of the generator identity~\eqref{eq:generator} has the form
$ \dot{\vb z}(t) \approx \mb{A} \vb{z}(t) + \mb{B} \vb{u}(t),$
where $\mb{A} \in \R^{N \times N}$, $\mb{B} \in \R^{N \times m}$, and $\vb{z}(t) = [\psi_1(\vb{x}(t)), \dots, \psi_N(\vb{x}(t))]^\top \in \R^N$ is the lifted state.
For such an approximation to be consistent, System~\eqref{eq:sys} must admit a, possibly infinite-dimensional, LTI Koopman representation. The following definition is an adaptation of the invariance conditions proposed in~\cite{iacob2024,haseli2026modeling}.\footnote{The definition might be relaxed for infinite-dimensional subspaces~$\mathbb{V}$ using \cite[Lemma~2.5.3]{curtain2012introduction} and Koopman invariance of suitably defined Hilbert spaces, see, e.g., \cite[Theorem~4.2]{kohne2024infty}, where Koopman invariance of suitably generated reproducible kernel Hilbert spaces is shown.}

\begin{definition}[LTI Koopman representation]
    \label{def:koopman_lti}
    A nonlinear system \eqref{eq:sys} admits an LTI Koopman representation w.r.t. a function space~$\V \subset \mathcal{C}^1(\R^n)$
    if for all $\psi \in \V$ and $\vb u \in \mathbb{R}^m$, the inclusion $\mathcal{L}_{\vb{u}}\psi \in \V$ holds and
    \begin{equation}\label{eq:Koopman:generator} 
        (\mathcal{L}_{\vb{u}} \psi)(\vb{x}) = (\mathcal{L}_0 \psi)(\vb{x}) + \langle \mathcal{B} \psi, \vb{u} \rangle,
    \end{equation}
    where $\mathcal{L}_0$ is the Koopman generator of the uncontrolled system, i.e., $\vb u(t) \equiv 0$, and $\mathcal{B} : \V \to \R^m$ 
    is a linear operator.
\end{definition}

From \Cref{def:koopman_lti} and equation~\eqref{eq:generator} for the Koopman generator for $\mathcal{C}^1$-functions,
it is clear that the inclusion $\langle \nabla \psi, \tilde{\vb{f}}(\cdot, 0)\rangle \in \V$ holds and $\langle \nabla \psi, \tilde{\vb{f}}(\cdot, \vb{u}) - \tilde{\vb{f}}(\cdot, 0)\rangle$ is linear in $\vb{u}$ for all $\psi \in \V$.
For our analysis, we are interested in function spaces that contain the coordinate maps,
which is a standard assumption in the literature~\cite{korda2018}. Note that, for a finite-dimensional lifting,
\Cref{ass:observables} is equivalent to the existence of $\mb C\in \R^{n\times N}$ with $\vb x = \mb C\vb z(\vb x)$.
\begin{assumption}[Coordinate maps]
    \label{ass:observables}
    Let $\V \subset \mathcal{C}^1(\R^n)$ be a function space that contains the coordinate maps, 
    i.e., $(\vb{x} \mapsto x_i) \in \V$ holds for all $i \in \{1, \dots, n\}$.
\end{assumption}

With \Cref{ass:observables}, we can immediately reason that System~\eqref{eq:sys} is affine in the inputs with constant input matrix~$\mb G$.
\begin{proposition}
    \label{cor:lin_input}
    Let \Cref{ass:observables} hold and System~\eqref{eq:sys} admit an LTI Koopman representation (see \Cref{def:koopman_lti}). Then, for System~\eqref{eq:sys},
    $\tilde{\vb{f}}(x,u) = \vb f(x) + \mb G\vb u$ for some $\vb f \in \mathcal{C}^1(\R^n,\R^n)$ and constant input matrix~$\mb{G} \in \R^{n \times m}$.
\end{proposition}
\begin{proof}
    Since, by \Cref{ass:observables}, $\psi_i \in \V$ with $\psi_i(\vb{x}) = x_i$, we have $\nabla \psi_i(\vb{x}) = \vb{e}^i$, i.e., the $i$-th unit 
    vector. 
    Now, the assumed LTI Koopman representation~\eqref{eq:Koopman:generator} yields 
    \begin{equation}\nonumber
        \langle\mathcal{B} \psi_i, \vb{u} \rangle = \langle \vb{e}^i, \tilde{\vb{f}}(\vb{x}, \vb{u}) - \tilde{\vb{f}}(\vb{x}, 0) \rangle = \tilde f_i(\vb{x}, \vb{u}) - \tilde f_i(\vb{x}, 0).
    \end{equation}
    Since $\langle\mathcal{B} \psi_i, \vb{u} \rangle$
    is linear in~$\vb{u}$ and does not depend on $\vb{x}$, 
    $\tilde{\vb{f}}$~is affine in the inputs with constant coefficients.
\end{proof}

\section{On LTI Koopman Representations}
\label{sec:main}

In this section, we present our main theorem, i.e., a characterization of the system structure admitting an LTI Koopman representation assuming directional excitability and \Cref{ass:observables}.
Then, we show that directional excitability is a weak controllability notion by providing three sufficient conditions. 

Since \Cref{cor:lin_input} uncovered a simpler structure of the nonlinear system, we can, without loss of generality, restrict the analysis to systems of the form
\begin{equation}
    \label{eq:sys_control_linear}
    \dot{\vb{x}}(t) = \vb{f}(\vb{x}(t)) + \mb{G}\vb{u}(t) = \vb{f}(\vb{x}(t)) + \sum\nolimits_{i=1}^m \vb{g}^i u_i(t),
\end{equation}
where $\vb{f}: \R^n \to \R^n$ is continuously differentiable, $\mb{G} \in \R^{n \times m}$ with full column rank, and $\vb g^i$, the $i$-th column of $\mb G$. 
\begin{assumption}[Directional excitability]
    \label{ass:excitability} 
    Let an invertible matrix $\mb T = [\vb{v}^1\, \dots\, \vb{v}^n] \in \R^{n \times n}$ exist with $\vb{v}^i = \vb{g}^i$ from~\eqref{eq:sys_control_linear} for $i \in \{1, \dots, m\}$, and, for $i \in \{m+1,\ldots,n\}$, $\vb{v}^i = \frac{\partial}{\partial \vb{x}}\vb{f}\left(\vb{x}^i\right) \vb{v}^j$ for some $\vb{x}^i\in \R^n$ and $j<i$.
\end{assumption}

Next, we present our main theorem, which states that under \Cref{ass:excitability}, an LTI Koopman representation exists if and only if the underlying control system is affine linear.
The proof is split into two parts. First, we show the key property that if an LTI Koopman representation exists, the function space only contains affine observables. From there, the assertion w.r.t.\ the system structure can be directly inferred. We refer to the appendix for the proof.
\begin{theorem}
    \label{thm:affine}
    Let \Cref{ass:excitability} hold. System~\eqref{eq:sys_control_linear} admits an LTI Koopman representation w.r.t.\ a function space~$\V$ satisfying \Cref{ass:observables} if and only if the system is affine linear, i.e., 
    $$
        \dot{\vb{x}}(t) = \mb{F}\vb{x}(t) + \mb{G}\vb{u}(t) + \vb{d}
    $$
    holds with $\mb{F}\in \R^{n\times n}$, $\mb{G} \in \R^{n \times m}$, and $\vb{d} \in \R^n$.
    Then, for each $\psi \in \V$, there exist coefficients $\alpha_0, \dots, \alpha_n \in \R$ such that, in addition, 
    $
        \psi(\vb{x}) = \alpha_0 + \sum\nolimits_{i=1}^n \alpha_i x_i.
    $, i.e., all $\psi \in \V$ are affine.
\end{theorem}
\begin{remark}
    \label{rem:diffeo}
    Instead of \Cref{ass:observables}, assume that the state can be reconstructed by means of a diffeomorphism~$S$~\cite{fan2026,iacob2026} as $\vb x = S^{-1}(\vb y)$ and let $\psi_i\in \V$ hold with $\psi_i(x) = (\vb{e}^i)^\top S(x)$ for $i\in \{1, \dots, n\}$. Now, \Cref{thm:affine} yields
    $$\dot{\vb y}  = \tfrac{\partial}{\partial {\vb x}}S(S^{-1}(\vb y))f(S^{-1}(\vb y), \vb u) = \mb{F}\vb{y} + \mb{G}\vb{u} + \vb{d}$$
    for suitable matrices $\mb F$, $\mb G$ and vector $\vb d$.
\end{remark}
\begin{remark}
    If System~\eqref{eq:sys_control_linear} does not satisfy \Cref{ass:excitability}, a matrix~$\mb T$ can still be constructed in the same way but with $\operatorname{rank}\mb T <n$. To admit an LTI Koopman representation, the dynamics along the subspace spanned by the columns of~$\mb T$ must be affine linear by \Cref{thm:affine}. Along the orthogonal complement of that subspace, the dynamics are autonomous but can be nonlinear; see~\cite{shang2026existence} for the 
    finite-dimensional discrete-time case.
\end{remark}
Our characterization suggests that, for a given system, successful application of LTI Koopman modeling also indicates the potential suitability of other linear techniques, including frequency domain~\cite{schoukens2020} and direct data-driven methods~\cite{shang2024}.

In the following, we provide sufficient conditions for \textit{directional excitability}. First, we show that if the control input can influence the state along arbitrary directions for a suitably chosen initial state, directional excitability holds.

\begin{proposition}\label{prop:detectability}
    Consider System \eqref{eq:sys_control_linear}. Let, for each direction $\vb{v} \in \R^n \setminus \{0\}$, exist an initial state~$\vb{x}^0$ and $\vb{u}^1, \vb{u}^2 \in L^\infty_{\operatorname{loc}}([0,\infty),\R^m)$ such that
    $$ 
        \langle\vb{v}, \vb{x}(t; \vb{x}^0, \vb{u}^1 ) \rangle \neq \langle\vb{v}, \vb{x}(t; \vb{x}^0, \vb{u}^2 ) \rangle 
    $$
    for some time $t > 0$.
    Then, System \eqref{eq:sys_control_linear} is directionally excitable, i.e., \Cref{ass:excitability} holds.
\end{proposition}
\begin{proof}
    We proof the statement by contraposition. Let $k< n$ be the largest possible number of independent vectors following the construction in \Cref{ass:excitability} and
    $$\mb{T}:= [\vb{v}^1, \dots, \vb{v}^{k}]\in \R^{n \times k}$$ 
    a corresponding realization. Since $\mb{T}$ cannot be extended, 
    the subspace spanned by the columns of~$\mb T$ is invariant under all linear maps of the form $\frac{\partial }{\partial \vb{x} }\vb{f}(\vb{x})$,
    i.e.,
    $$\tfrac{\partial }{\partial \vb{x} }\vb{f}(\vb{x})\vb v \in \operatorname{span}\{\vb{v}^1, \dots, \vb{v}^{k}\}$$
    for all $\vb v \in \operatorname{span}\{\vb{v}^1, \dots, \vb{v}^{k}\}$ and $\vb x \in \R^n$.
    Hence, there exists a matrix-valued continuous function $\mb{\Lambda}: \R^n \to \R^{k\times k}$ such that for all $\vb{\alpha}\in \R^{k}$ and all $\vb{x}\in \R^n$,
    $$\tfrac{\partial }{\partial \vb{x} }\vb{f}(\vb{x}) \mb{T} \vb{\alpha} = \mb{T} \mb{\Lambda}(\vb{x}) \vb{\alpha}.$$
    Now, let $\vb{x}(t; \vb{x}^0, \vb{u})$ be the solution on a time interval $[0, T)$ for some initial state $\vb{x}^0$ and $\vb{u}\in L^\infty_{\operatorname{loc}}([0,\infty),\R^m)$. Further, let $\bar{\vb{u}}\in L^\infty_{\operatorname{loc}}([0,\infty),\R^m)$ and $\epsilon\in \R$. The variation of the solution is given by $\delta \vb{x}(t) := \frac{d}{d\epsilon} \vb{x}(t; \vb{x}^0, \vb{u} + \epsilon \bar{\vb{u}}) |_{\epsilon = 0}$ and it can be shown that $\delta \vb{x}$ is the solution to the initial value problem
    \begin{align*}
        \tfrac{d}{dt}\delta \vb{x}(t) = \tfrac{\partial}{\partial \vb{x}} \vb{f}(\vb{x}(t; \vb{x}^0, \vb{u})) \delta \vb{x}(t) + \mb{G} \bar{\vb{u}}(t),\quad
        \delta \vb{x}(0) = 0.
    \end{align*}
    We use the ansatz $\delta \vb{x}(t) = \mb{T} \vb{\alpha}(t)$, where $\vb{\alpha}: [0, T) \to \R^{k}$ and find
    \begin{align*}
        \mb{T} \dot{\vb{\alpha}}(t) &= \tfrac{\partial}{\partial \vb{x}} \vb{f}(\vb{x}(t; \vb{x}^0, \vb{u})) \mb{T} \vb{\alpha}(t) + \mb{G} \bar{\vb{u}}(t)\\
        &= \mb{T} \mb{\Lambda}(\vb{x}(t; \vb{x}^0, \vb{u})) \vb{\alpha}(t) + \mb{T} [\vb{e}^1, \dots, \vb{e}^m]  \bar{\vb{u}}(t).
    \end{align*}
    Since $\mb{T}$ has a left inverse, we can conclude that 
    $$\dot{\vb{\alpha}}(t) = \mb{\Lambda}(\vb{x}(t; \vb{x}^0, \vb{u})) \vb{\alpha}(t) + [\vb{e}^1, \dots, \vb{e}^m]  \bar{\vb{u}}(t), \quad \vb{\alpha}(0) = 0,$$
    which is a standard inhomogeneous linear initial value problem that has a unique solution.
    Hence, $\delta \vb{x}(t) = \mb{T} \vb{\alpha}(t)$ for all $t \in [0, T)$.
    Finally, note that because $\mb{T}$ has only $k < n$ columns, there exists a vector $\vb{c} \in \R^n\setminus \{0\}$ such that $\vb{c}^\top \mb{T} = 0$. Hence, $\vb{c}^\top \delta \vb{x}(t) = 0$ for all $t \in [0, T)$, which concludes the proof.
\end{proof}

\Cref{ass:excitability} also holds
if, at some state, the linearization of System~\eqref{eq:sys_control_linear} is controllable as shown in \Cref{prop:linearization}.
\begin{proposition}
    \label{prop:linearization}
    Let, for some $\vb{x^0} \in \R^n$, the linearization of System~\eqref{eq:sys_control_linear}, $\dot{\vb{y}}(t) = \mb F \vb{y}(t) + \mb{G} \vb{u}(t)$,
    with $\mb F := \tfrac{\partial}{\partial \vb{x}} \vb{f}(\vb{x}^0)$ be controllable, i.e., $\operatorname{rank} \big[ \mb{G}, \mb F\mb{G}, \dots, \mb F^{n-1} \mb{G} \big] = n$, 
    then \Cref{ass:excitability} holds.
\end{proposition}
\begin{proof}
    The first $m$ columns of the matrix $\mb T$ in \Cref{ass:excitability} are given by the columns $\vb g^1, \dots, \vb g^m$ of~$\mb G$. Now, we can iterate as follows. Let $\vb v^1, \dots, \vb v^k$ be independent vectors satisfying the construction from \Cref{ass:excitability}. If $k=n$, the construction is complete. Otherwise,
    we can find $i \in \{1, \dots, k\}$ such that $\vb v^{k+1}:= \mb F \vb v^i$ is not in the span of $\vb v^1, \dots, \vb v^k$. 
    Otherwise,
    $\operatorname{rank} \big[ \mb{G}, \mb F\mb{G}, \dots, \mb F^{n-1} \mb{G} \big] = k < n$ holds --~a contradiction to the assumed controllability.
\end{proof}

Finally, we connect directional excitability to the Lie algebra rank condition~(LARC; \cite{coron2009}).
\Cref{def:lie_algebra,def:larc} recapitulate the definitions of a Lie algebra and of LARC, respectively.

\begin{definition}[Lie algebra]
    \label{def:lie_algebra}
    Let $\mathcal M\subset \mathcal C^\infty(\R^n, \R^n)$ be a set of differentiable maps, and let
    $[\cdot, \cdot]:  \mathcal C^\infty(\R^n, \R^n) \times \mathcal C^\infty(\R^n, \R^n) \to \mathcal C^\infty(\R^n, \R^n)$ with
    $$[\vb f, \vb h](\vb x):=\tfrac{\partial}{\partial \vb x}\vb h(\vb x)\vb f(\vb x) - \tfrac{\partial}{\partial \vb x}\vb f(\vb x)\vb h(\vb x)$$
    denote the Lie bracket.
    The Lie algebra generated by $\mathcal{M}$, denoted $\operatorname{Lie}(\mathcal M)$, is the smallest linear subspace $V$ of $\mathcal C^\infty(\R^n, \R^n)$ for which the inclusion $\mathcal{M}\subset V$ holds, and $\vb f \in V$ and $\vb h \in V$ implies $[\vb f, \vb h ]\in  V $.
\end{definition}

\begin{definition}[LARC]
    \label{def:larc}
    For System~\eqref{eq:sys_control_linear}, where we assume smooth drift dynamics~$\vb f \in \mathcal C^\infty(\R^n, \R^n)$, we define the set of maps
    $ \mathcal{M} := \{\vb f, (\vb x \mapsto \vb g^1), \dots, (\vb x \mapsto \vb g^m) \}.$
    Let $\operatorname{Lie}(\mathcal M)$ denote the Lie algebra generated by~$\mathcal M$, and let $$\mathcal{A}(\vb{x}^0):= \{\vb h (\vb x^0) \mid  \vb h \in \operatorname{Lie}(\mathcal M)\}\subset \R^n.$$
    System \eqref{eq:sys_control_linear} is said to satisfy LARC at $\vb{x}^0$ if $\mathcal{A}_{\mathrm{ LA}}(\vb{x}^0) = \R^n$.
\end{definition}
\Cref{prop:lie_algebra} shows that if System~\eqref{eq:sys_control_linear} admits an LTI Koopman representation w.r.t. to a function space satisfying \Cref{ass:observables}, LARC at a controlled equilibrium implies directional excitability. 
\begin{proposition}
    \label{prop:lie_algebra}
    Let System~\eqref{eq:sys_control_linear} admit an LTI Koopman representation w.r.t. to a function space satisfying \Cref{ass:observables}. Further, let LARC hold at a controlled equilibrium $(\vb{x}^0, \vb{u}^0)$.
    Then, \Cref{ass:excitability} holds.
\end{proposition}
\begin{proof}
    From \eqref{eq:generator}, \Cref{def:koopman_lti,lem:finV}, it follows that, for all $i\in\{1,\dots, n\}$, $\vb x \in \R^n$, and $\vb u \in \R^m$,
    $\langle\nabla f_i(\vb x),\mb G\vb u\rangle \equiv \langle \mathcal B f_i,\vb u\rangle,$
    where $\mathcal B f_i\in \R^m$.
    Let $j\in\{1,\dots,m\}$, we find with $\vb u =\vb e^j$, the $j$-th unit vector, 
    $$\langle\nabla f_i(\vb x),\mb G\vb e^j\rangle = \langle\nabla f_i(\vb x),\vb g^j\rangle\equiv\alpha_{ij}$$
    for some $\alpha_{ij}\in\R$, and hence,    
    $\frac{\partial}{\partial \vb{x}}\left( \frac{\partial}{\partial \vb{x}} \vb f (\vb{x})\vb g^i \right) \equiv 0$. Thus, 
    \begin{align*}
        \mathcal A(\vb{x}^0) = &\spanop \Big\{ \vb f\left(\vb{x}^0\right), \vb g^1, \dots, \vb g^m, \\
        \tfrac{\partial}{\partial \vb{x}} &\vb f\left(\vb{x}^0\right) \vb g^1, \dots, \big(\tfrac{\partial}{\partial \vb{x}} \vb f\left(\vb{x}^0\right)\big)^2 \vb g^1, \dots \Big\}= \R^n.
    \end{align*}
    Now, due to the Cayley-Hamilton Theorem and the observation that $\vb f\left(\vb{x}^0\right) + \mb G \vb u^0 = 0$, \Cref{ass:excitability} can be fulfilled by an appropriate choice of spanning vectors of $\mathcal A(\vb{x}^0)$, as in the proof of \Cref{prop:linearization}.
\end{proof}

\section{LTI Modeling Bias}
\label{sec:bias}

This section describes how the LTI structure introduces a bias, a structural error, into the modeling of nonlinear control systems.
We identify how the bias is caused by properties of System~\eqref{eq:sys_control_linear} and the choice of observables~$\psi_1, \dots , \psi_N$.

First, we find in \Cref{prop:bias1} that a bias is introduced if an observable has curvature along an input direction, i.e.,
\begin{eqnarray}
    \label{eq:check_psi}
    \tfrac{\partial^2}{\partial \vb{x}^2}\psi_i \mb{G} \neq 0
\end{eqnarray}
for some $i \in \{1, \dots, N\}$.

\begin{proposition}
    \label{prop:bias1}
    For System \eqref{eq:sys_control_linear}, let the lifting be defined as $\vb z(\vb x):= [\psi_1(\vb x), \dots , \psi_N(\vb x)]^\top$ with $\psi_i \in \mathcal{C}^2(X)$ for $i \in \{1, \dots, N\}$, and let $X \subset \R^n$ be a convex and bounded set. Further, let $\mb{A} \in \R^{N \times N}$ and $\mb{B} \in \R^{N \times m}$ be the matrices of the LTI Koopman approximation $\dot{\vb{z}} \approx \mb{A}\vb{z}+ \mb{B}\vb u$.
    And, assume there exists~$\vb x^0 \in X$ with $\mb B = \frac{\partial}{\partial \vb x}\vb z(\vb x^0)\mb G$.
    Then, it holds
    for an observable $\psi(\vb x) = \vb a^\top \vb z(\vb x)$ with $\vb{a} \in \R^N$,
    \begin{align*}
        \left|\mathcal L_{\vb u}\psi(\vb x) - \vb a^\top (\mb{A}\vb{z} + \mb{B}\vb u)\right| \le 
        \left| \mathcal L_0 \psi(\vb x)- \vb a^\top \mb{A}\vb{z}(\vb x) \right| \\+ \diam(X) \sup_{\vb \xi \in X} \left\|\tfrac{\partial^2}{\partial \vb{x}^2}\psi(\vb \xi) \mb{G} \right\| \|\vb{u}\|
    \end{align*}
    for all $\vb x \in X$ and $\vb u \in \R^m$.
\end{proposition}
\begin{proof}
    First, we decompose the approximation error in an autonomous and an input-dependent part as follows:
    \begin{align*}
        &|\mathcal L_{\vb u}\psi(\vb x) - \vb a^\top (\mb{A}\vb{z} + \mb{B}\vb u)| \\
        &\le \left| \mathcal{L}_0 \psi(\vb x) - \vb a^\top \mb{A}\vb{z}(\vb x) \right| + \left| \tfrac{\partial}{\partial \vb{x}} \psi(\vb x) \mb{G} \vb u - \vb a^\top \mb{B} \vb u \right|.
    \end{align*}
    Now, by Taylor's Theorem, for all $\vb x_0, \vb x \in X$, there exists~$\vb\xi \in X$ 
    such that
       $$ \tfrac{\partial}{\partial \vb{x}} \psi(\vb x) \mb{G} \vb u = \tfrac{\partial}{\partial \vb{x}} \psi(\vb x_0) \mb{G} \vb u + (\vb x - \vb x_0)^\top \tfrac{\partial^2}{\partial \vb{x}^2}\psi(\vb \xi) \mb{G} \vb u,$$
    and, hence,
    \begin{align*}
        | &\tfrac{\partial}{\partial \vb{x}} \psi(\vb x) \mb{G} \vb u - \vb a^\top \mb{B} \vb u | \\ &\le \left| \tfrac{\partial}{\partial \vb{x}} \psi(\vb x_0) \mb{G} \vb u - \vb a^\top \mb{B} \vb u \right| + \left| (\vb x - \vb x_0)^\top \tfrac{\partial^2}{\partial \vb{x}^2}\psi(\vb \xi) \mb{G} \vb u\right|.
    \end{align*}
    Taking the infimum over $\vb x^0 \in X$ and the supremum over $\vb \xi \in X$ and noting that
    \begin{align*}
        \inf_{\vb x^0 \in X} \left\| \tfrac{\partial}{\partial \vb{x}} \psi(\vb x^0) \mb{G} - \vb a^\top \mb{B} \right\| &\le \|\vb a^\top(\tfrac{\partial}{\partial \vb x}\vb z(\bar{\vb x})\mb G -\mb B)\| = 0
    \end{align*}
    yields the result.
\end{proof}

Note that the assumption in \Cref{prop:bias1} can be easily satisfied by setting $\mb B := \frac{\partial}{\partial \vb x}\vb z(\vb x^0)\mb G$. If, instead, $\mb B$~is the result of solving a minimization problem as in~\cite{korda2018}, the additional error can still assumed to be small.

\begin{example} 
    \label{ex:sms}
    Consider the
    actuated slow-manifold system
\begin{equation*}
    \frac{d}{dt }
    \begin{bmatrix}
        x_1 \\
        x_2
    \end{bmatrix}
    = \vb f (\vb x) + \vb g u=
    \begin{bmatrix}
        x_1 - x_2^2\\
        x_2
    \end{bmatrix} +
    \begin{bmatrix}
        1 \\
        0
    \end{bmatrix}
    u
    ,
\end{equation*}
where the input enters the first component.
It is easy to see that it admits a finite-dimensional LTI Koopman representation
w.r.t. the function space spanned by $\psi_1(\vb x) = x_1$, $\psi_2(\vb x) = x_2$, and $\psi_3(\vb x) = x_2^2$ as shown in~\cite{proctor2018generalizing}.
The only observable with non-zero curvature is $\psi_3$, for which
$\tfrac{\partial^2}{\partial \vb{x}^2}\psi_3 \vb g = 0.$
Note, however, that the input never influences $x_2$, and, therefore, the system is not directionally excitable.
\end{example}

\begin{figure}[t]
    \centering
    \input{figures/brunton_error}
    \caption{Mean EDMDc approximation error for \Cref{ex:sms_excitable} for varying polynomial degree of observables.}
    \label{fig:edmd_training_error_brunton}
\end{figure}
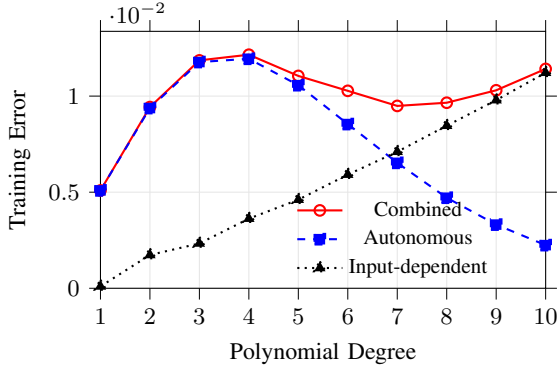

While the autonomous approximation error in \Cref{prop:bias1} can generally be reduced by incorporating more expressive observables, the input-dependent approximation error might increase at the same time, as demonstrated in \Cref{ex:sms_excitable}.
\begin{example}
    \label{ex:sms_excitable}
    If, compared to 
    \Cref{ex:sms}, the input enters the second component, i.e., 
    \begin{equation*}
        \frac{d}{dt}
        \begin{bmatrix}
            x_1 \\
            x_2
        \end{bmatrix}
        = \vb f (\vb x) + \vb g u=
        \begin{bmatrix}
            x_1 - x_2^2\\
            x_2
        \end{bmatrix}
        +
        \begin{bmatrix}
            0 \\
            1
        \end{bmatrix}
        u,
    \end{equation*}
    the system is directionally excitable by~\Cref{prop:detectability}. Hence, by \Cref{thm:affine}, it cannot admit an LTI Koopman representation.
    With the same observables as in \Cref{ex:sms}, we find that
    $ \tfrac{\partial^2}{\partial \vb{x}^2}\psi_3 \vb g =[0,2]^\top.$
    This choice of observables would therefore lead to an input-dependent approximation error.\\
    For data-based LTI Koopman modeling,\footnote{We use extended Dynamic Mode Decomposition with control (EDMDc)~\cite{proctor2016a,
    korda2018} with monomial observables, 10,000 samples of the state and input, drawn uniformly at random from $X = [-1, 1]^n$ and $U = [-0.1, 0.1]^m$. The continuous-time nonlinear system is discretized with a sampling time of $0.01$. To calculate the error, we scale the observables by their mean value over the training data to avoid scaling dependencies, calculate the maximum error at each sample, and average the result.}
    \Cref{fig:edmd_training_error_brunton} shows the approximation error for varying maximal polynomial degree.
    We observe that, after an initial increase, the error of the autonomous part decreases with higher polynomial degrees, while the error of the input-dependent part increases. This is in line with the theoretical results from \Cref{prop:bias1}. In particular, since the second derivatives of polynomials scale linearly with their degree, the input-dependent error increases linearly with the degree as well.
\end{example}

To reduce the input-dependent approximation error, one may restrict the function space~$\V$ to observables that are affine in the input directions, i.e., $\frac{\partial^2}{\partial \vb{x}^2}\psi \mb{G} \equiv 0$ for all~$\psi \in~\V$, which was proposed in~\cite{asada2024}.
Next, we show
that this restriction introduces another bias stemming from
\begin{itemize}
    \item curvature of the drift dynamics\footnote{Here and in the following analysis, we assume that the drift dynamics are twice continuously differentiable.} along the input directions, i.e., there exists $j \in \{1,\dots,m\}$ such that
    \begin{eqnarray}
        \label{eq:check_f}
        \tfrac{\partial^2}{\partial \vb x^2} f_j \mb G \neq 0,
    \end{eqnarray}
    \item dynamics-dependent coupling of the second derivatives of the observables with the input directions, i.e., there exists $i \in \{1, \dots, N\}$ such that
    \begin{eqnarray}
        \label{eq:check_coupling}
        \tfrac{\partial^2}{\partial \vb{x}^2}\psi_i \tfrac{\partial}{\partial \vb x} \vb f \mb G \neq 0.
    \end{eqnarray}
\end{itemize}

 To state \Cref{prop:bias2},
\Cref{ass:restricted_V} formalizes the restriction of the function space~$\V$.
It specifies that for all functions $\varphi \in \mathcal{C}^2(\R^n)$, we find~$\widehat \varphi\in\V$ that represents~$\varphi$ perfectly along the non-input directions conditioned on the imposed restriction.

\begin{assumption}
    \label{ass:restricted_V}
    Let~$\mb G\in \R^{n\times m}$ with full column rank be given.
    Let $\V \subset \mathcal{C}^2(\R^n)$ satisfy $\frac{\partial^2}{\partial \vb{x}^2}\psi \mb{G} \equiv 0$ 
    for all $\psi \in \V$.
    Given~$\varphi\in \mathcal{C}^2(\R^n)$, there exists a function~$\widehat \varphi\in\V$ with the following properties for all $\vb x^0 \in (\ran \mb G)^\perp$ and $\vb x \in \R^n$:
    \begin{enumerate}[(P1)]
        \item $\widehat{\varphi}(\vb x^0) = \varphi(\vb x^0)$ \vspace{0.1cm}
        \item $\tfrac{\partial}{\partial \vb x} \widehat{\varphi}(\vb x^0) = \tfrac{\partial}{\partial \vb x} \varphi(\vb x^0)$ \vspace{0.1cm}
        \item $(\mb I - \mb G \mb G^\dagger)^\top \big(\tfrac{\partial^2}{\partial \vb x^2} \widehat{\varphi}(\vb x) - \tfrac{\partial^2}{\partial \vb x^2} \varphi(\vb x)\big)(\mb I - \mb G \mb G^\dagger) = 0$\\[0.1cm] 
        with pseudoinverse $\mb G^\dagger := (\mb{G}^\top \mb{G})^{-1} \mb{G}^\top$ of $\mb{G}$. 
    \end{enumerate}
\end{assumption}

\begin{proposition}
    \label{prop:bias2}
    Let the system be given by \eqref{eq:sys_control_linear} and let 
    $\V\subset \mathcal{C}^2(\R^n)$ 
    satisfy \Cref{ass:restricted_V}.
    Further, let 
    $X\subset \R^n$ be a convex and bounded set, $\vb u \in \R^m$, $\psi \in \V$, and $\widehat{\varphi} \in \V$ the approximation of~$\mathcal{L}_{\vb{u}} \psi$ as in \Cref{ass:restricted_V}.
    Then,
    \begin{align*}
        &|\mathcal{L}_{\vb{u}} \psi(\vb x) - \widehat \varphi(\vb x) | \le \diam(X)^2 \|\mb G^\dagger\| \sup_{\vb x \in X} \left\|\tfrac{\partial^2}{\partial \vb{x}^2}\mathcal{L}_{\vb{u}} \psi(\vb x) \mb G \right\|
    \end{align*}
    for all $\vb x \in X$, and
    \begin{align}
        \nonumber
        \tfrac{\partial^2}{\partial \vb x^2} \mathcal{L}_{\vb{u}} \psi(\vb x) \mb G = &\sum_{i=1}^{n} \tfrac{\partial}{\partial x_i}\psi(\vb x)\tfrac{\partial^2}{\partial \vb x^2} f_i(\vb x) \mb G\\
        \label{eq:curvature}
        &+ \tfrac{\partial^2}{\partial \vb x^2} \psi(\vb x) \tfrac{\partial}{\partial \vb x} \vb f(\vb x) \mb G.
    \end{align}
\end{proposition}
\begin{proof}
    Let $\vb x \in X$ and $\vb x^0 \in (\ran \mb G)^\perp$ be given and $\varphi := \mathcal{L}_{\vb{u}} \psi$. Then, with \Cref{ass:restricted_V} (P1) and (P2), Taylor series expansion yields that there exists~$\vb \xi \in X$ satisfying
    \begin{align*}
        \varphi(\vb x) -& \widehat \varphi(\vb x) \\
        =&\varphi(\vb x^0) - \widehat \varphi(\vb x^0)+\big( \tfrac{\partial}{\partial \vb x} \varphi(\vb x^0) - \tfrac{\partial}{\partial \vb x} \widehat \varphi(\vb x^0) \big)(\vb x - \vb x^0)\\
        &+ \tfrac{1}{2} (\vb x - \vb x^0)^\top \big(\tfrac{\partial^2}{\partial \vb x^2} \varphi(\vb \xi) -\tfrac{\partial^2}{\partial \vb x^2} \widehat\varphi(\vb \xi)\big) (\vb x - \vb x^0)\\
        =&\tfrac{1}{2} (\vb x - \vb x^0)^\top \big(\tfrac{\partial^2}{\partial \vb x^2} \varphi(\vb \xi) -\tfrac{\partial^2}{\partial \vb x^2} \widehat\varphi(\vb \xi)\big) (\vb x - \vb x^0).
    \end{align*}
    Moreover, by \Cref{ass:restricted_V} (P3)
    \begin{align*}
        &\tfrac{\partial^2}{\partial \vb x^2} \varphi -\tfrac{\partial^2}{\partial \vb x^2} \widehat\varphi \\
        &= (\mb G \mb G^\dagger + \mb I - \mb G \mb G^\dagger)^\top\big(\tfrac{\partial^2}{\partial \vb x^2} \varphi -\tfrac{\partial^2}{\partial \vb x^2} \widehat\varphi \big)(\mb G \mb G^\dagger + \mb I - \mb G \mb G^\dagger)\\
        &= (\mb G \mb G^\dagger)^\top \tfrac{\partial^2}{\partial \vb x^2} \varphi (2 \mb I - \mb G \mb G^\dagger).
    \end{align*}
    Therefore, we can bound the approximation error as follows:
    \begin{align*}
        |&\varphi(\vb x) - \widehat \varphi(\vb x)| \\&\le \tfrac{1}{2} \|\vb x - \vb x^0\|^2 \sup_{\vb x \in X} \left\|(\mb G \mb G^\dagger)^\top \tfrac{\partial^2}{\partial \vb x^2} \varphi(\vb x) (2 \mb I - \mb G \mb G^\dagger) \right\| \\
        &\le \diam(X)^2 \|\mb G^\dagger\| \sup_{\vb x \in X} \left\|\tfrac{\partial^2}{\partial \vb{x}^2}\varphi(\vb x) \mb G \right\|.
    \end{align*}
    Computing the second derivative of $\varphi$ while exploiting that $\frac{\partial^2}{\partial \vb{x}^2}\psi \mb{G} \equiv 0$ yields the mentioned expression~\eqref{eq:curvature}.    
\end{proof}

\begin{example}
    For the system defined in \Cref{ex:sms_excitable}, we have that
    $$ \tfrac{\partial^2}{\partial \vb x^2} f_1(\vb x)\vb g =
    \begin{bmatrix}
        0 & 0 \\
        0 & -2
    \end{bmatrix} 
    \begin{bmatrix}
            0 \\
            1
        \end{bmatrix} = \begin{bmatrix}
            0 \\
            -2
        \end{bmatrix}
    \neq 0,$$
    indicating that the $x_1$ dynamics have curvature that cannot be captured by observables that are affine along the input direction.
\end{example}

\begin{example}
    \label{ex:triple}
    The three-dimensional system
    \begin{equation*}
        \frac{d}{dt}
        \begin{bmatrix}
            x_1 \\
            x_2 \\
            x_3
        \end{bmatrix} = \vb f (\vb x) + \vb g u=
        \begin{bmatrix}
            x_2^2 \\
            x_3 \\
            0
        \end{bmatrix}
        + 
        \begin{bmatrix}
            0\\
            0 \\
            1
        \end{bmatrix}u
        ,
    \end{equation*}
    is directionally excitable, and, hence, does not admit an LTI Koopman representation including the coordinate maps. When we choose $\psi_1(\vb x) = x_1$, $\psi_2(\vb x) = x_2$, $\psi_3(\vb x) = x_3$, and $\psi_4(\vb x) = x_2^2$,
    we find $\tfrac{\partial^2}{\partial \vb{x}^2}\psi_i \vb g = 0$ for all $i \in \{1, \dots, 4\}$ and $\tfrac{\partial^2}{\partial \vb x^2} f_j \vb g = 0$ for all $j \in \{1, 2, 3\}$. However, the coupling of the input and non-input directions through the dynamics results in non-zero curvature of $\mathcal L_{\vb u} \psi_4$ along the input direction since
    \begin{align*}
        \tfrac{\partial^2}{\partial \vb{x}^2}\psi_4(\vb x) &\tfrac{\partial}{\partial \vb x} \vb f(\vb x) \vb g = 
        [0,2,0]^\top
        \neq 0.
    \end{align*}
\end{example}

From a practical point of view, 
it is advisable to include observables that evolve almost linearly with the inputs, such that~\eqref{eq:check_psi} and~\eqref{eq:check_coupling} become approximate equalities. Moreover, LTI Koopman modeling is likely to fail if the system has strong inherent nonlinear input dependencies leading to~\eqref{eq:check_f}.

\begin{remark}
    The results and conditions from this section can be analogously derived for the discrete-time case for systems of the form $\vb x^+ = \vb f(\vb x) + \mb G \vb u$ and discrete-time LTI Koopman approximations $\vb z^+ \approx \mb A \vb z + \mb B \vb u$. 
\end{remark}

\section{Conclusion}

We have shown that directionally excitable non-affine systems cannot admit an LTI Koopman representation including the coordinate maps.
We analyzed the bias introduced by the LTI structure and derived conditions that can be checked to assess the suitability of an LTI Koopman approximation.
Our analysis suggests that LTI Koopman modeling for nonlinear control systems should always be accompanied by an analysis of the modeling bias, following the presented conditions.

\appendix
\label{sec:proof}

\noindent\textbf{Proof of Theorem~\ref{thm:affine}.} First, we proof that \Cref{ass:observables} implies the inclusion $f_i \in \V$, $i \in \{1, \dots, n\}$, for the drift dynamics of System~\eqref{eq:sys_control_linear}.
\begin{lemma}
    \label{lem:finV}
    Let System~\eqref{eq:sys_control_linear} admit an LTI Koopman representation w.r.t.\ a function space~$\V$ satisfying \Cref{ass:observables}. Then, $f_i \in \V$ holds for all $i \in \{1, \dots, n\}$.
\end{lemma}
\begin{proof}
    Let an index $i \in \{1,\ldots,n\}$ be given. Then, \Cref{ass:observables} yields $\psi \in \V$ for $\psi(\vb{x}) = x_i$. It follows that
    $$ 
        f_i = \langle\vb{e}^i, \vb{f}\rangle = \langle\nabla \psi, \vb{f}\rangle \in \V
    $$
    Since the index~$i$ was arbitrary, the assertion follows. 
\end{proof}
Next, we iteratively show 
that all observables in~$\V$ are affine along the column vectors of the matrix~$\mb T$ from \Cref{ass:excitability}.
\begin{lemma}
    \label{lem:prop_VT}
    Let a system given by \eqref{eq:sys_control_linear} and satisfying \Cref{ass:excitability} admit an LTI Koopman representation w.r.t. a function space $\V$ satisfying \Cref{ass:observables}.
    Let $\mb T$ be defined as in \Cref{ass:excitability}, and let 
    $\V_\mb{T} := \{ \varphi\in \mathcal C^1(\R^n) \mid \exists\, \psi \in \V\ \forall\; \vb{y} \in \R^n: \varphi(\vb y) = \psi(\mb{T}\vb{y}) \}$
    be the space of transformed observables and $\vb{h}(\vb{y}) := \mb{T}^{-1}\vb{f}(\mb{T}\vb{y})$. The following holds:
    \begin{enumerate}[(i)]
        \item For all $\varphi \in \V_\mb{T}$, $\langle \nabla \varphi, \vb{h}\rangle \in \V_\mb{T}$.
        \item For all $i\in \{1, \dots, n\}$, $(\vb{y}\mapsto y_i) \in \V_\mb{T}$.
        \item For all $i\in \{1, \dots, n\}$, $h_i \in \V_\mb{T}$.
        \item For all $\varphi \in \V_\mb{T}$, $\varphi$ is affine in the first $m$ coordinates. That is, there exist constants $\alpha_1, \dots, \alpha_m \in \R$ such that
        $$\varphi(y) = \varphi(0, \dots, 0, y_{m+1}, \dots, y_n) + \sum\nolimits_{i=1}^m \alpha_i y_i.$$
        \item For all $\varphi \in \V_\mb{T}$, $\varphi$ is affine. That is, there exist constants $\alpha_0, \dots, \alpha_n \in \R$ such that
        $$\varphi(\vb{y}) = \alpha_0 + \sum\nolimits_{i=1}^n \alpha_i y_i.$$
        \end{enumerate}
\end{lemma}
\begin{proof}
    (\textit{i})-(\textit{iii}) The statements follow directly from the definition of $\V_\mb{T}$, \Cref{def:koopman_lti,ass:observables,lem:finV}.

    (\textit{iv}) From \eqref{eq:generator} and \Cref{def:koopman_lti}, it follows that, for all $\psi \in \V$, $\vb x \in \R^n$, and $\vb u \in \R^m$,
    $\langle\nabla \psi(\vb x),\mb G\vb u\rangle \equiv \langle \mathcal B \psi,\vb u\rangle,$
    where $\mathcal B \psi\in \R^m$.
    Let $i\in\{1,\dots,m\}$, we find with $\vb u =\vb e^i$, the $i$-th unit vector, and $\psi(\vb x) = \varphi(\mb T \vb y)$ for some $\varphi \in \V_{\mb T}$ that
    there exists $\alpha_i \in \R$ such that
    \begin{align*}
        \alpha_i &\equiv \langle \nabla\psi(\mb{T}\vb{y}) , \mb G\vb e^i\rangle = \langle \nabla\psi(\mb{T}\vb{y}) , \mb T\vb e^i\rangle\\
        &= \langle \nabla \varphi(\vb{y}), \vb{e}^i\rangle = \tfrac{\partial }{\partial y_i}\varphi(\vb{y}).
    \end{align*}
    Since $\V$ and $\V_{\mb T}$ are isomorphic, the assertion follows for all $\varphi \in \V_{\mb T}$.
    
    (\textit{v}) We proof the statement by induction.
    From (\textit{iv}), we know that $\varphi$ is affine in the first $m$ coordinates.
    Now, let the statement hold for $k<n$~($\ast$), i.e, for all $\varphi \in \V_\mb{T}$ there are constants $\alpha_1, \dots, \alpha_k \in \R$ such that
    \begin{equation}
        \tag{$\ast$}
        \label{eq:induction}
         \varphi(\vb{y}) = \varphi(0, \dots, 0, y_{k+1}, \dots, y_n) + \sum\nolimits_{i=1}^k \alpha_i y_i.
    \end{equation}
    Therefore, we find
    \begin{align*}
        \langle \nabla \varphi, \vb{h} \rangle
        = \sum_{i=1}^{n}\tfrac{\partial}{\partial y_i} \varphi h_i
        = \underbrace{\sum_{i=1}^{k}\alpha_i h_i}_{\in \V_\mb{T}} +  \sum_{i=k+1}^{n}\tfrac{\partial}{\partial y_i} \varphi h_i.
    \end{align*}
    Since $\langle \nabla \varphi, \vb{h} \rangle \in \V_\mb{T}$ by (\textit{i}), it follows from (\textit{iii}) and the fact that~$\V_{\mb T}$ is a vector space that
    $$ \tilde \varphi :=\sum\nolimits_{i=k+1}^{n}\tfrac{\partial}{\partial y_i} \varphi h_i \in \V_\mb{T}.$$
    By \Cref{ass:excitability}, there exists $j \leq k$ such that $\vb{v}^{k+1} = \frac{\partial}{\partial \vb{x}} \vb{f}(\vb{x}^{k+1}) \vb{v}^j$ for some $\vb{x}^{k+1} \in \R^n$, and by ($\ast$), there exists $\beta \in \R$ such that
    $ \tfrac{\partial}{\partial y_j} \tilde \varphi(\vb{y}) \equiv \beta$. Expanding the derivative 
    yields
    \begin{align*}
        \beta = \sum_{i=k+1}^{n} (\underbrace{\tfrac{\partial^2}{\partial y_i \partial y_j} \varphi}_{=0}  h_i + \tfrac{\partial}{\partial y_i} \varphi \underbrace{\tfrac{\partial}{\partial y_j} h_i}_{=:a_i \in \R})
        = \sum_{i=k+1}^{n} a_i \tfrac{\partial}{\partial y_i} \varphi.
    \end{align*}
    With $\vb{a} := [a_{1}, \dots, a_n]^\top$, there exists $\gamma \in \R$ such that
    $$\langle \nabla \varphi, \vb{a} \rangle = \sum\nolimits_{i=1}^{k} a_i \tfrac{\partial}{\partial y_i} \varphi + \sum\nolimits_{i=k+1}^{n} a_i \tfrac{\partial}{\partial y_i} \varphi = \gamma,$$
    where the first sum is constant due to ($\ast$). Finally, note that 
    \begin{align*}
        \vb{a} &= \tfrac{\partial}{\partial y_j} \vb{h} 
        = \tfrac{\partial}{\partial y_j}\mb{T}^{-1}\vb{f}(\mb{T}\vb{y})
        = \mb{T}^{-1} \tfrac{\partial}{\partial \vb{x}}\vb{f}(\mb{T}\vb{y}) \mb{T}\vb{e}^j\\
        &= \mb{T}^{-1} \tfrac{\partial}{\partial \vb{x}}\vb{f}(\vb{x}^{k+1})\vb{v}^j 
        = \mb{T}^{-1}\vb{v}^{k+1}
        = \vb{e}^{k+1},
    \end{align*}
    and we conclude that $\tfrac{\partial}{\partial y_{k+1}} \varphi = \langle \nabla \varphi, \vb{e}^{k+1} \rangle = \gamma$. Thus, by induction, the statement holds for all $k \in \{1, \dots, n\}$.
\end{proof}

The main theorem, stating that only affine-linear systems can yield an LTI Koopman representation, now follows as a corollary from the last lemma.

\begin{proof}[\textbf{Proof of \Cref{thm:affine}}]
    ($\Rightarrow$) We are given $\V$ and $\V_\mb{T}$ as in \Cref{lem:prop_VT}.
    Since $\V$ and $\V_\mb{T}$ are isomorphic, for $\psi\in\V$, there is exactly one $\varphi\in \V_\mb{T}$ such that for all $\vb{x}\in \R^n$,
    $ \psi(\vb{x}) = \varphi(\mb{T}^{-1}\vb{x}).$
    Substituting the affine structure from \Cref{lem:prop_VT}~(\textit{v}), we get
    $ \psi(\vb{x})=\alpha_0 + \vb{a}^\top \mb{T}^{-1}\vb{x},$
    with $\alpha_0\in\R$ and $\vb{a} \in \R^n$. Hence, $\psi$ is affine. By \Cref{lem:finV}, we can apply this result to $f_i\in\V$ and obtain the mentioned system structure.
    ($\Leftarrow$) With $\V:=\spanop\{(\vb{x} \mapsto 1),(\vb{x} \mapsto x_1),\dots,(\vb{x} \mapsto x_n)\}$, the assertion follows immediately.
\end{proof}

\balance

\bibliographystyle{IEEEtran}
\bibliography{references}

\end{document}

%% file: figures/brunton_error.tex
\begin{tikzpicture}

\begin{axis}[
    width=7.5cm,
    height=5cm,
    xlabel={Polynomial Degree},
    ylabel={Training Error},
    xmin=1, xmax=10,
    ymin=0,
    xtick={1,2,3,4,5,6,7,8,9,10},
    grid=major,
    grid style={gray!20},
    tick align=outside,
    tick style={black},
    legend style={
        font=\footnotesize,
        draw=none,
        fill=none,
        at={(0.9,0.)},
        anchor=south east,
        cells={align=left}
    },
    label style={font=\small},
    tick label style={font=\small},
]

\addplot[
    red,
    thick,
    mark=o,
    mark size=2pt
]
table[
    col sep=comma,
    x=Degree,
    y={Training Error}
]{data/Brunton_edmd_errors.csv};
\addlegendentry{Combined}

\addplot[
    blue,
    thick,
    dashed,
    mark=square*,
    mark size=2pt
]
table[
    col sep=comma,
    x=Degree,
    y={Autonomous Part Error}
]{data/Brunton_edmd_errors.csv};
\addlegendentry{Autonomous}

\addplot[
    black,
    thick,
    dotted,
    mark=triangle*,
    mark size=2.5pt
]
table[
    col sep=comma,
    x=Degree,
    y={Control Part Error}
]{data/Brunton_edmd_errors.csv};
\addlegendentry{Input-dependent}

\end{axis}

\end{tikzpicture}